\documentclass{article}
\usepackage{indentfirst,latexsym,bm}
\usepackage{amsthm}
\usepackage{amsfonts}
\usepackage[dvips]{graphics}
\begin{document}
\title{\bf Effect of a crossing change on crossing number
\thanks{Research supported by NMOE.}}
\author{\sl Longting Wu,\quad Shuting Shao$^{\dag}$,\quad Shan Liu,\quad Fengchun Lei\\[5pt]
\small School of Mathematical Sciences, Dalian University of
Technology\\
\small Dalian, Liaoning, 116024, P.R.China\\
\small $^{\dag}$shaoshuting@mail.dlut.edu.cn}
\date{}
\maketitle

\begin{abstract}
The purpose of this article is to give a preliminary clarification
on the relation between crossing number and crossing change.

With a main focus on the span of $X$ polynomial, we prove that, as
our theorem claims, the crossing number of the link after crossing
change can be estimated when certain conditions are met.

At the end of the article, we give an example to demonstrate a
special case for the theorem and a counterexample to explain that
the theorem cannot be applied if the obtained link is not
alternating.
\end{abstract}


\section{Introduction}
When you project a knot or link $L$ in $\mathbb{R}^3$ onto a plane
and keep account of which part of $L$ goes over and which goes under
at any crossing. Then it is natural for you to asked how the link
$L$ (throughout this paper, a knot will be considered a link of one
component) is changed by a \textit{crossing change}, i.e. reversing
which goes over and which goes under at one of the crossing point.

Martin Scharlemann \cite{scharlemann} has done an expository survey
about the role that the simple operation of changing a crossing has
played in knot theory, discussing topics such as unknotting number,
Dehn surgery, sutured manifolds, 4-manifold topology, etc. Aside
from these, the nugatory crossing conjecture proposed by Xiao-Song
Lin \cite{lin} also turned out to be an important focus concerning
the issue of crossing change. But limited work has been done on the
issues of influence of crossing change to \textit{crossing number},
noted as $c(L)$ which is the least number of crossings that occur in
any projection of the link. In this article, we will mainly discuss
the change of crossing number when the situation of changing a link
by one crossing change is applied. More specifically, we prove

\theoremstyle{definition}
\newtheorem{theorem}{Theorem}[section]
\begin{theorem}
If you change a reduced and alternating projection of a connected link $L$
into a projection of alternating link $\tilde{L}$ by a crossing change, then it holds
\[c(\tilde{L})\leq c(L)-2.\]
\end{theorem}

In Preliminaries, we are going to introduce some concepts and
theorems which are related and helpful to the proof of this theorem.
We mainly make use of the results of Kauffmann, Murasugi and
Thistlethwaite which characterize alternating links by relating the
spread of the Jones polynomial to crossing number and Dasbach, Lin
and Stoimenow's work on the coefficients of the Jones polynomial to
prove this theorem.

At the end of this article, we give a counterexample to demonstrate
that this theorem cannot be applied when $\tilde{L}$ is a
non-alternating link, and we give some suggestions on further work
in this direction.

\section{Preliminaries}
We say an \textit{alternating link} is a link which has a projection
in which over- and under-crossings alternate when you travel a
circuit around each component of the link.

Call a projection of link is \textit{reduced} if there are no
obvious removed crossings (See Figure 1).

\begin{center}
\hspace{16pt}\scalebox{0.3}{\includegraphics*[500,200]{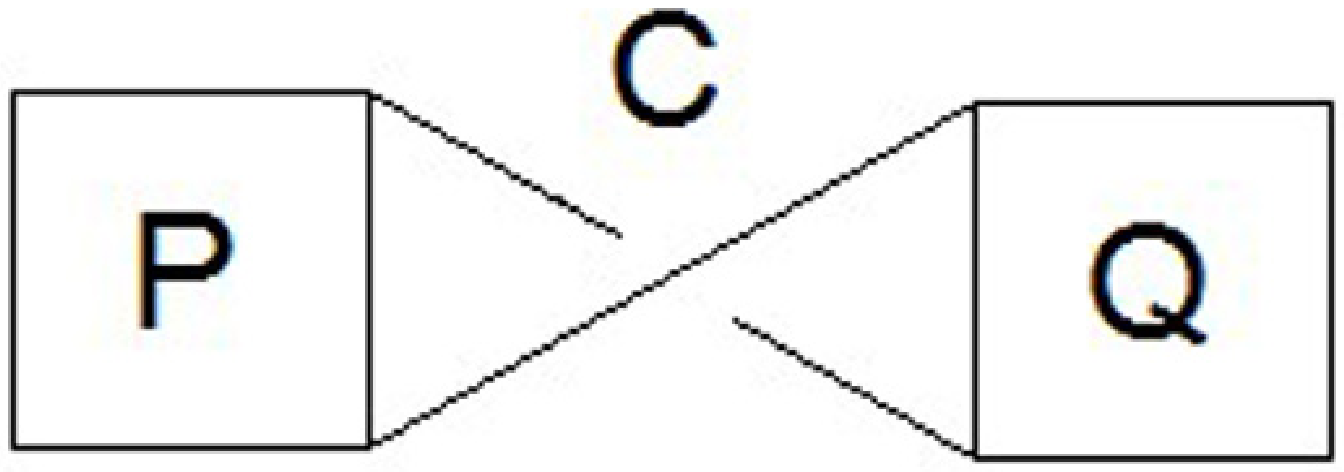}}\\
\textbf{Figure 1}
\end{center}

A link diagram $D$ is \textit{split}, if there is a closed curve not
intersecting it, but which contains parts of the diagram in both its
in- and exterior(See Figure 2).

\begin{picture}(200,90)
\put(128,15){\scalebox{0.3}{\includegraphics*[420,250]{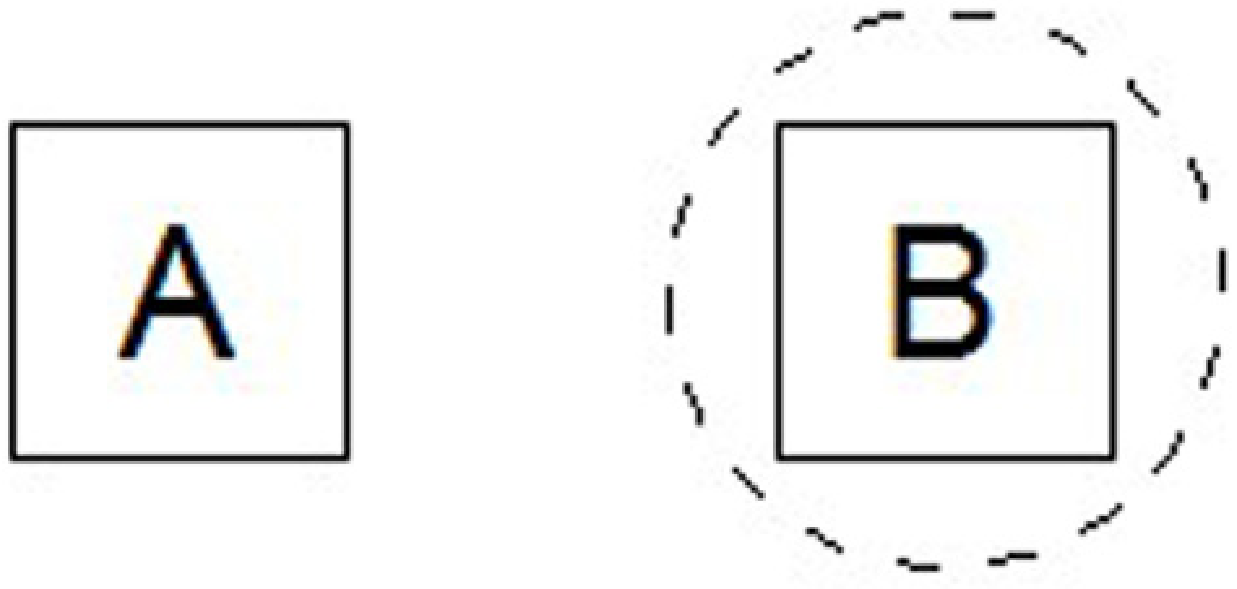}}}\hspace{30pt}\\
\put(172,5){\textbf{Figure 2}}
\end{picture}

Otherwise, $D$ is \textit{connected} or \textit{non-split}. A link
is \textit{split} if it has a split diagram, and otherwise
\textit{connected}.

Now that every concepts in the content of the main theorem is
introduced, let us introduce other background definitions and
lemmas, which is critical in the proof of the theorem.

We know that the Kauffmann bracket $<D>$ of a knot or link diagram
$D$ is a Laurent polynomial in a variable $A$, obtained by summing
over all the states $S$ terms
\[A^{a(S)-b(S)}(-A^{2}-A^{-2})^{|S|-1}\]
where a \textit{state} of $S$ is a choice of \textit{splicings} (or
\textit{splittings}) of type $A$ or $B$ for any single crossing(See
Figure 3).

\begin{picture}(250,80)
\put(80,30){\scalebox{0.65}{\includegraphics*[360,60]{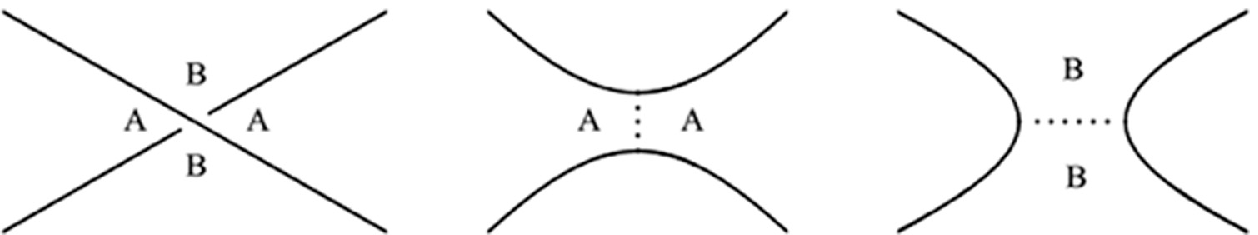}}}
\put(55,-3){\shortstack[l]{\textbf{Figure 3}: the left figure is the
A- and B-corners of a crossing, the \\center one is type A
splitting, and the right one is type B splitting.}}
\end{picture}
\vspace{11pt}

$a(S)$ and $b(S)$ denote the number of type $A$ (resp. type $B$)
splicings and $|S|$ the number of (disjoint) circles obtained after
all splicings in $S$.

Also, we know that Kauffmann bracket $<D>$ is an invariant under
Reidemeister II and III moves. If $D$ is oriented, assign a value of
$+1$ or $-1$ to each crossing according to the usual right-hand
rule(See Figure 4).

\begin{center}
\hspace{10pt}\scalebox{0.3}{\includegraphics*[400,260]{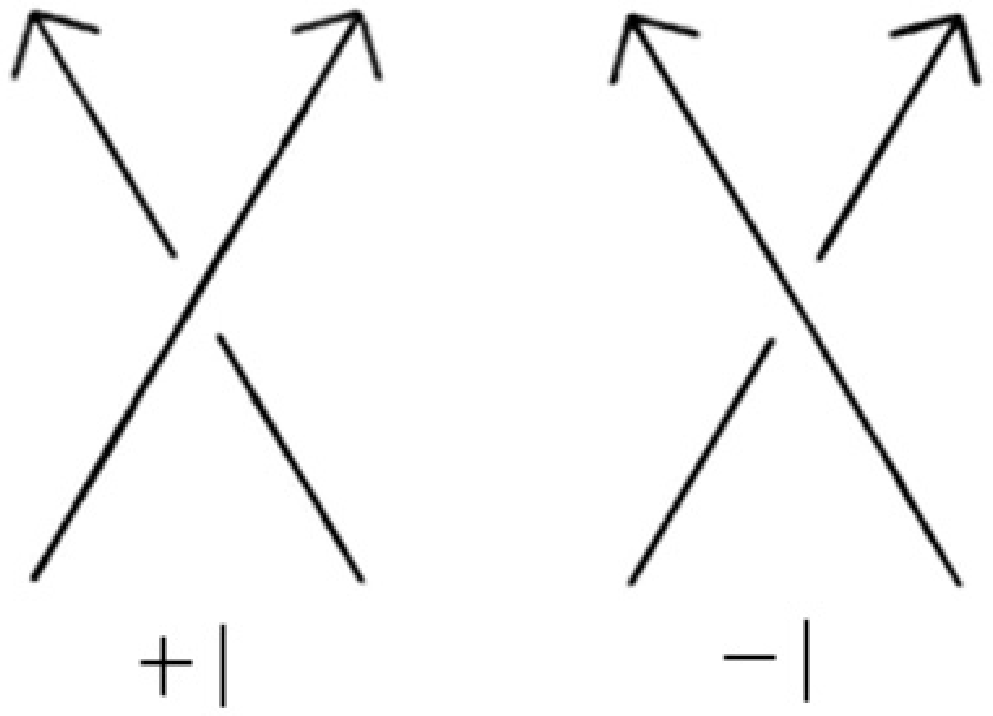}}\\
\textbf{Figure 4}
\end{center}

Define the \textit{writhe} of $D$, $\omega(D)$, to be the sum of
these values. So we can get the \textit{X polynomial}
\begin{equation}
X(L)=\sum_S (-A^3)^{-\omega(D)} A^{a(S)-b(S)}(-A^{2}-A^{-2})^{|S|-1}
\end{equation}
and Jones polynomial
\[V_L(t)=X(L)|_{A=t^{-1/4}}.\]

As we know, Jones polynomial satisfies the original skein relation
\cite{jones}
\[t^{-1}V(L_+)-tV(L_-)+(t^{-1/2}-t^{1/2})V(L_0)=0\]
where $L_+$, $L_-$, $L_0$ be three oriented link projections that
are identical except where they appear as in Figure 5.

\begin{center}
\hspace{25pt}\scalebox{0.4}{\includegraphics*[460,200]{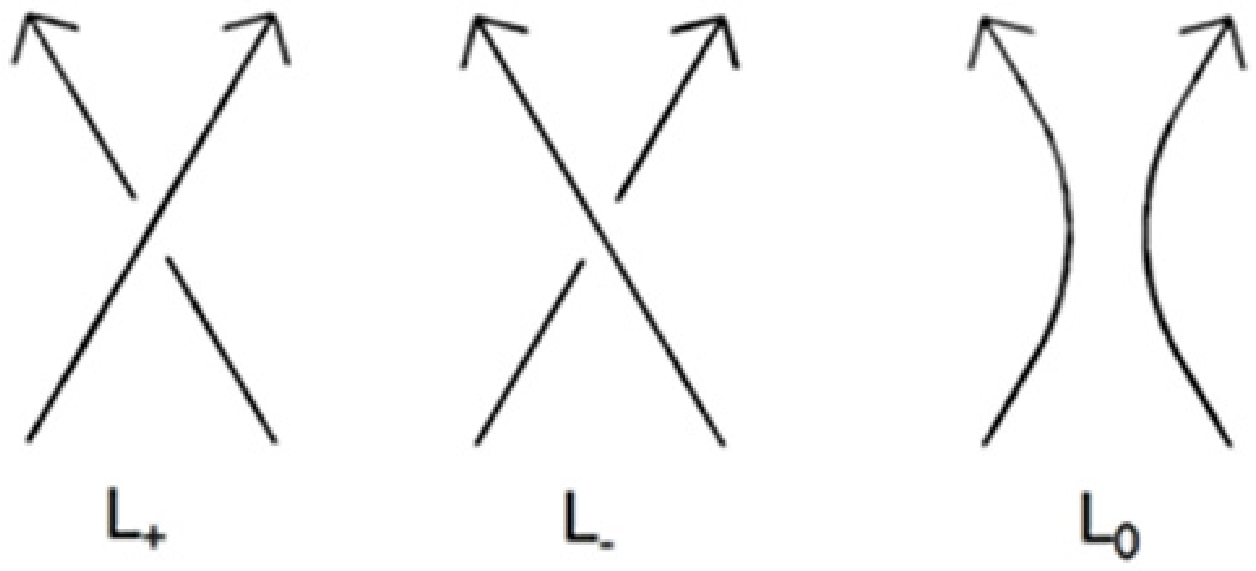}}\\
\textbf{Figure 5}
\end{center}

So that, $X(L)$ satisfies
\begin{equation}
A^4X(L_+)-A^{-4}X(L_-)+(A^2-A^{-2})X(L_0)=0.
\end{equation}

Let \textit{span}($L$) denote the difference
\[\mbox{span}(L)=\mbox{maxdeg}\ X(L)-\mbox{mindeg}\ X(L).\]

\newtheorem{definition}{Definition}[section]

And let
\[X(L)=a_0A^{l}+a_1A^{l+4}+\cdots+a_mA^{l+4m}\]
with $a_0\neq0$ and $a_m\neq0$ is the X polynomial of a knot or link
$L$. Through out the paper, we will notation $X_i=a_i$ and
$\widehat{X}_i=a_{m-i}$ for the $(i+1)$-st or $(i+1)$-last
coefficient of $X(L)$.

We then have
\newtheorem{lemma}[definition]{Lemma}
\begin{lemma} \cite{Mur} \cite{Ka}   \cite{thist87}
For a connected and alternating link $L$, it holds
\begin{enumerate}
\item $X_0,\,\widehat{X}_0=\pm 1$.

\item The signs of the coefficients in $X(L)$ are alternating.

\item
\begin{equation}
\mbox{span}(L)=4c(L).
\end{equation}

\item If L is a non-connected and alternating link with $n(L)$
components, then
\[\mbox{span}(L)=4c(L)+4(n(L)-1).\]
\end{enumerate}
\end{lemma}

Now let's consider the $2$-st or $2$-last coefficient of $X(L)$.

Define $S_A$ to be the $A$-state (resp. $B$-state) the state in which all crossings are $A$-spliced,
and $S_B$ is defined analogously.

\begin{definition}\cite{sto}
{\upshape \textit{$A$-graph} $G(A(D))=G(S_A)$ of $D$ is defined as
the planar graph with vertices given by loops in the $A$-state of
$D$, and edges given by crossings of $D$. The analogous terminology
is set up also for the \textit{$B$-graph} $G(B(D))$\,(See Figure
6).}
\end{definition}

\begin{picture}(300,120)
\put(50,0){\scalebox{0.35}{\includegraphics*[800,320]{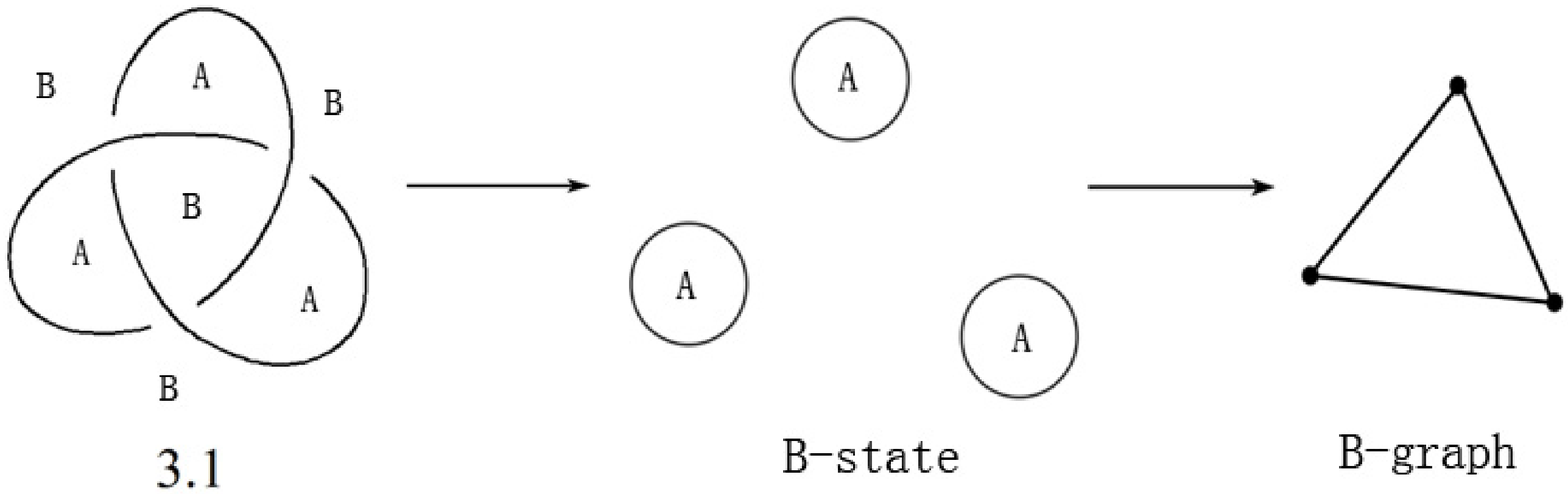}}}
\put(170,5){\textbf{Figure 6}}
\end{picture}

Let $v(G)$ and $e(G)$ be the number of vertices and edges of a graph $G$. Let $G'$ be $G$ with multiple edges removed (so
that a simple edge remains). We call $G'$ the \textit{reduction} of $G$.

Now we have
\begin{lemma} \cite{DL}
If $D$ is an alternating projection of an alternating link $L$, then we have
\[\begin{array}{lll}
|\widehat{X}_1| & = & e(G(A(D))')-v(G(A(D))')+1\\
        |{X}_1| & = & e(G(B(D))')-v(G(B(D))')+1.
  \end{array}\]
\end{lemma}

\begin{proof}
With the purpose of according with proof of the main result in this
article, we will only give the full version of proof for the second
equation here and the first equation can be proved similarly.

For $V_L(t)=X(L)|_{A=t^{-1/4}}$, we only need to consider the
coefficient of the second highest degree term of $V_L(t)$. It is
well known, e.g. \cite{bol}, that the Jones Polynomial of an
alternating link $L$ satisfies:
\[V_L(t)=(-1)^{\omega(D)}t^{(|S_A|-|S_B|+3\omega(D))/4}T_{G(B(D))}(-t,-1/t).\]

Here $T_{G(B(D))}(x,y)$ is the Tutte polynomial of $G(B(D))$. Then,
we only need to consider the coefficient of the second highest
degree term of $T_{G(B(D))}(-t,-1/t)$.

In \cite{thist}, we know that
\[T_{G}(-t,-1/t)=\sum_{\tilde{F}\subseteq \tilde{E}}\left((-t-1)^{k(\tilde{F})-1}\left(-\frac{1}{t}-1\right)^{|\tilde{F}|-|V|
+k(\tilde{F})} \prod_{e\in \tilde{F}}P(\mu(e))\right).\]

Here $G=G(B(D))=(V,E)$ and $G(B(D))'=(V,\tilde{E})$. While
$k(\tilde{F})$ is the number of components of the graph
$(V,\tilde{F})$, $\mu(e)$ is the number of edges in $G$ that are
parallel to $e$ and $P(m)$ is defined as
\[P(m)\quad :=\quad 1-t^{-1}+t^{-2}-\cdots \pm t^{-m+1}.\]

Thus, we know that the coefficient of the second highest degree term
of $T_{G(B(D))}(-t,-1/t)$ is $(-1)^{|V|-1}(|V|-1- |\tilde{E}|)$.
Then the absolute value of the coefficient of the second highest
degree term of $V_L(t)$ is $|\tilde{E}|-|V|+1$.

Then after substitution of notations, we have
\[|{X}_1|  =  e(G(B(D))')-v(G(B(D))')+1.\]
\end{proof}

\section{Proof of the main result}
With all the background knowledge introduced above, let us go back
to Theorem 1.1 mentioned in the beginning of the article.

\setcounter{section}{1}
\begin{theorem}
If you change a reduced and alternating projection of a connected
link $L$ into a projection of alternating link $\tilde{L}$ by a
crossing change, then it holds
\[c(\tilde{L})\leq c(L)-2.\]
\end{theorem}
\begin{proof}
Any given orientation of each component of $L$, we denote by $P$ a
reduced and alternating projection of the oriented link $L$. And we
suppose the value of the crossing point we choose is $-1$ (See
Figure 4). Meanwhile, let us assume that $L_-=P$. Then it can be
easily to see that $L_+$ is a projection of the $\tilde{L}$. Next,
let us start discussion on the change of span of $X$ polynomial.

First, from\,(1), we have
\[\begin{array}{lll}
         X(L_0) & = & \sum_{S_0} (-A^3)^{-\omega(L_0)} A^{a(S_0)-b(S_0)}(-A^{2}-A^{-2})^{|S_0|-1}\\
         X(L_-) & = & \sum_{S_-} (-A^3)^{-\omega(L_-)} A^{a(S_-)-b(S_-)}(-A^{2}-A^{-2})^{|S_-|-1}
    \end{array}\]
where $S_0$ and $S_-$ are statuses corresponding to $L_0$ and $L_-$.
Meanwhile, $\omega(L_0)$ and $\omega(L_-)$ are writhes of $L_0$ and
$L_-$, respectively.

Also, from\,(2), we know $X(L_+)$, $X(L_0)$ and $X(L_-)$ satisfy
\[A^4X(L_+)-A^{-4}X(L_-)+(A^2-A^{-2})X(L_0)=0.\]

Then
\begin{equation}
X(L_+)=A^{-8}X(L_-)+(A^{-6}-A^{-2})X(L_0).
\end{equation}

We denote by $m_-$ the lowest degree of $X(L_-)$ and $M_-$ the
highest degree. The corresponding coefficients of the term are
$X_0^{-}$ and $\widehat{X}_0^{-}$. The notations for $X(L_0)$
($m_0$, $M_0$, $X_0^{0}$ and $\widehat{X}_0^{0}$) can be define
analogously.

For $L_-$ is an alternating projection, it is easy to see that $L_0$
is also an alternating projection. And it is well known that
\cite{Ka}
\[\begin{array}{lll}
  m_- & = & -3\omega(L_-)-b(S_B^-)-2|S_B^-|+2\\
  m_0 & = & -3\omega(L_0)-b(S_B^0)-2|S_B^0|+2
  \end{array}\]
where $S_B^-$ and $S_B^0$ are $B$-states of $L_-$ and $L_0$,
respectively.

Then
\[\begin{array}{ll}
& \mbox{mindeg}(A^{-8}X(L_-))-\mbox{mindeg}((A^{-6}-A^{-2})X(L_0))\\
= &  -8+m_--(-6+m_0)\ =\ m_--m_0-2\\
=& 3(\omega(L_0)-\omega(L_-))+b(S_B^0)-b(S_B^-)+(|S_B^0|-|S_B^-|)-2.
\end{array}\]
And it is easy to know that
$\omega(L_-)=\omega(L_0)-1,b(S_B^-)=b(S_B^0)+1,|S_B^-|=|S_B^0|$.
Thus,
\begin{equation}
\mbox{mindeg}(A^{-8}X(L_-))=\mbox{mindeg}((A^{-6}-A^{-2})X(L_0)).
\end{equation}

And the coefficient of the lowest degree term of $A^{-8}X(L_-)$ is
$X_0^{-}$ and the coefficient of the lowest degree term of
$(A^{-6}-A^{-2})X(L_0)$ is $X_0^{0}$, then
\begin{equation}
X_0^{-}=(-1)^{(-\omega(L_-)+|S_B^-|-1)}=(-1)^{(-(\omega(L_0)-1)+|S_B^0|-1)}=-X_0^{0}.
\end{equation}

Then with\,(4),(5),(6), we have
\begin{equation}
\mbox{mindeg}(X(L_+))\geq \mbox{mindeg}(A^{-8}X(L_-))+4=m_--4.
\end{equation}

Now, let us consider the highest degree term of $X(L_+)$.

From\,(3), we have
\begin{equation}
\mbox{maxdeg}(A^{-8}X(L_-))-\mbox{mindeg}(A^{-8}X(L_-))= 4c(L_-).
\end{equation}

Also, for $L_-$ is a reduced and alternating projection of a
connected link, then $L_0$ must also be an alternating projection of
a connected link, then with\,(3)
\begin{equation}
\mbox{maxdeg}((A^{-6}-A^{-2})X(L_0))-\mbox{mindeg}((A^{-6}-A^{-2})X(L_0))=4c(L_0)+4.
\end{equation}

Therefore, by\,(5), (8), (9), we have
\begin{equation}
\mbox{maxdeg}(A^{-8}X(L_-))-\mbox{maxdeg}((A^{-6}-A^{-2})X(L_0))=4(c(L_-)-c(L_0))-4.
\end{equation}

\begin{enumerate}
\item If $L_0$ is already reduced, then $c(L_-)=c(L_0)+1$, by\,(10),
we have
\[\mbox{maxdeg}(A^{-8}X(L_-))=\mbox{maxdeg}((A^{-6}-A^{-2})X(L_0)).\]
Meanwhile, by Lemma 2.1, the coefficient of the highest degree term
of $A^{-8}X(L_-)$ is $\widehat{X}_0^{-}=X_0^{-}(-1)^{c(L_-)}$, the
coefficient of the highest degree term of $(A^{-6}-A^{-2})X(L_0)$ is
$-\widehat{X}_0^{0}=-X_0^{0}(-1)^{c(L_0)}=X_0^{0}(-1)^{c(L_-)}$.

Together with\,(4), (6), we have
\begin{equation}
\mbox{maxdeg}(X(L_+))\leq \mbox{maxdeg}(A^{-8}X(L_-))-4=M_--12.
\end{equation}

Then with\,(7), (11), we have
\[\mbox{span}(L_+)\leq \mbox{span}(L_-)-8.\]

Recall that $L_+$ is alternating link (not necessarily connected),
then by Lemma 2.1 $c(L_+)\leq \frac{1}{4}\mbox{span}(L_+$), thus
\[c(L_+)\leq \frac{1}{4}\mbox{span}(L_+)\leq \frac{1}{4}(\mbox{span}(L_-)-8)=c(L_-)-2.\]

\item If $L_0$ is not reduced, then $c(L_-)> c(L_0)+1$, by (10), we
have
\[\mbox{maxdeg}(A^{-8}X(L_-))>\mbox{maxdeg}((A^{-6}-A^{-2})X(L_0)).\]
Then with\,(4), we have
\begin{equation}
\mbox{maxdeg}(X(L_+))=M_--8.
\end{equation}

We denote the coefficient of the second lowest degree term of
$X(L_-)$ and $X(L_0)$ by $X_1^{-}$ and $X_1^{0}$, respectively.
Then, the coefficient of the second lowest degree term of
$A^{-8}X(L_-)$ and $(A^{-6}-A^{-2})X(L_0)$ are $X_1^{-}$ and $
X_1^{0}-X_0^{0}$, respectively. From\,(5), (6), we know that the sum
of the coefficient of the lowest degree terms of $A^{-8}X(L_-)$ and
$(A^{-6}-A^{-2})X(L_0)$ is $0$. Then,\,from (4), we know that the
coefficient of $(m_--4)$ degree term of $X(L_+)$ is
$X_1^{-}+X_1^{0}-X_0^{0}$.

Recall that the signs of the coefficients are alternating in Lemma
2.1, then
\[|X_1^{-}|=-X_0^{-}X_1^{-}\qquad |X_1^{0}|=-X_0^{0}X_1^{0}.\]
Together with\,(6), we have
\begin{equation}
|X_1^{-}+X_1^{0}-X_0^{0}| = ||X_1^{-}|-|X_1^{0}|-1|.
\end{equation}

Denote $B$-graphs of $L_-$ and $L_0$ by $G(B)_-$ and $G(B)_0$,
respectively. And $G(B)'_-$ and $G(B)'_0$ are reduction of
$B$-graphs of $L_-$ and $L_0$. Since $L_-$ and $L_0$ are both
alternating projections, then by Lemma 2.3, we have
\begin{eqnarray}
     |X_1^{-}| &=& e(G(B)'_-)-v(G(B)'_-)+1\\
     |X_1^{0}| &=& e(G(B)'_0)-v(G(B)'_0)+1.
\end{eqnarray}

For we obtain both $G(B)_-$ and $G(B)_0$ by $B$-splitting any single
crossing of $L_-$, then
\begin{equation}
v(G(B)'_-)=v(G(B)_-)=v(G(B)_0)=v(G(B)'_0).
\end{equation}

For we obtain $L_0$ from splitting open $c_0$ in $L_-$, and $c_0$ is
an edge $e_0$ that connects vertices $P$ and $Q$ in $G(B)_-$, then
we have $G(B)_0=G(B)_--e_0$ by definitions in Graph Theory. Let $n$
be the number of edge(s) connecting $P$ and $Q$ in $G(B)_-$. And now
let us do further discussion on $n$.

\begin{enumerate}
\item If $n=1$. \vspace{11pt}

Since $c_0$ is split open in $L_0$, then in $G(B)'_0$, there is no
edge connecting $P$ and $Q$. Meanwhile, the connections for other
vertices in $G(B)'_0$ are the same with ones of $G(B)'_-$.

Thus, $e(G(B)'_-)=e(G(B)'_0)+1$. With\,(14), (15) and (16), we have
$|X_1^{-}|=|X_1^{0}|+1$.

And with\,(13), we know that the coefficient of $(m_--4)$ degree
term of $X(L_+)$ is $0$. Thus,
\[\mbox{mindeg}(X(L_+))\geq
(m_--4)+4=m_-.\]

With\,(12), we have
\[
c(L_+)\leq \frac{1}{4}\mbox{span}(L_+)\leq
\frac{1}{4}(\mbox{span}(L_-)-8)=c(L_-)-2.
\]

\item If $n=2$ (See Figure 7).

\begin{center}
\scalebox{0.35}{\includegraphics*[270,200]{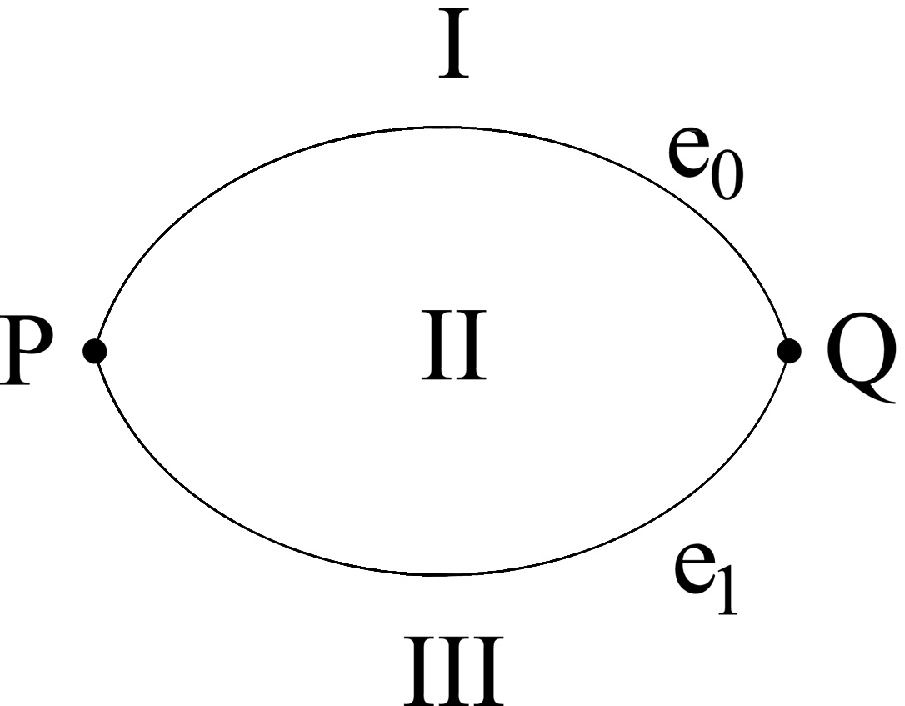}}\\
\textbf{Figure 7}
\end{center}

Since $L_0$ is not reduced, then in $G(B)_0$ exists an edge $e$,
which belongs to both region $1$ and region $2$. And these two
regions are connected (See Figure 8).

\begin{center}
\scalebox{0.4}{\includegraphics*[300,110]{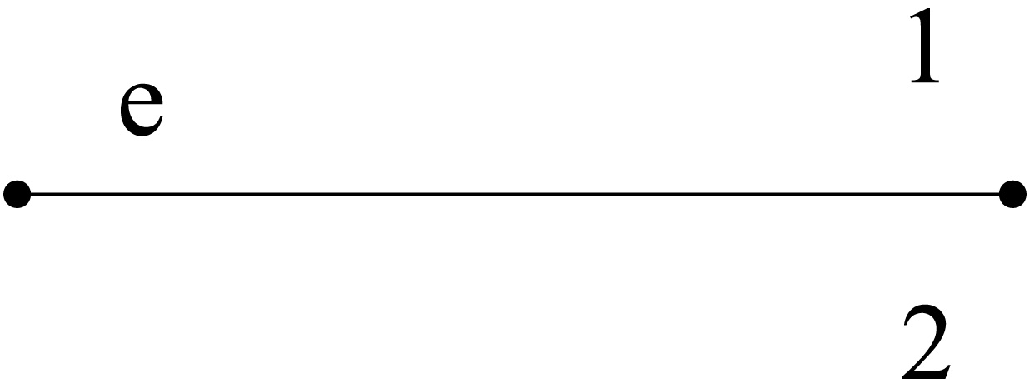}}\\
\textbf{Figure 8}
\end{center}

Since $G(B)_0=G(B)_--e_0$, $e$ can be regarded as an edge of
$G(B)_-$. Recalling that $L_-$ is reduced, then regions $1$ and $2$
are not connected in $G(B)_-$. On the other hand, from Figure 7, we
know that after splitting open $c_0$, the unconnected regions I and
II in $G(B)_-$ connect in $G(B)_0$. Meanwhile, the connectedness of
other regions in $G(B)_-$ maintains in $G(B)_0$. Thus, region $1$
and region $2$ must be I and II and $e$ must be an edge of region II
in $G(B)_-$. And with Figure 7, we know that the only edges belonged
to II in $G(B)_-$ are $e_0$ and $e_1$. However, there is no $e_0$ in
$G(B)_0$. Thus, $e=e_1$. With $e_1$ being an edge of region III in
$G(B)_-$, there must be $\mbox{I}=\mbox{III}$, namely I and III is
connected in $G(B)_-$. Then, $L_-$ must be like what is shown below
(See Figure 9).

\begin{center}
\scalebox{0.4}{\includegraphics*[450,170]{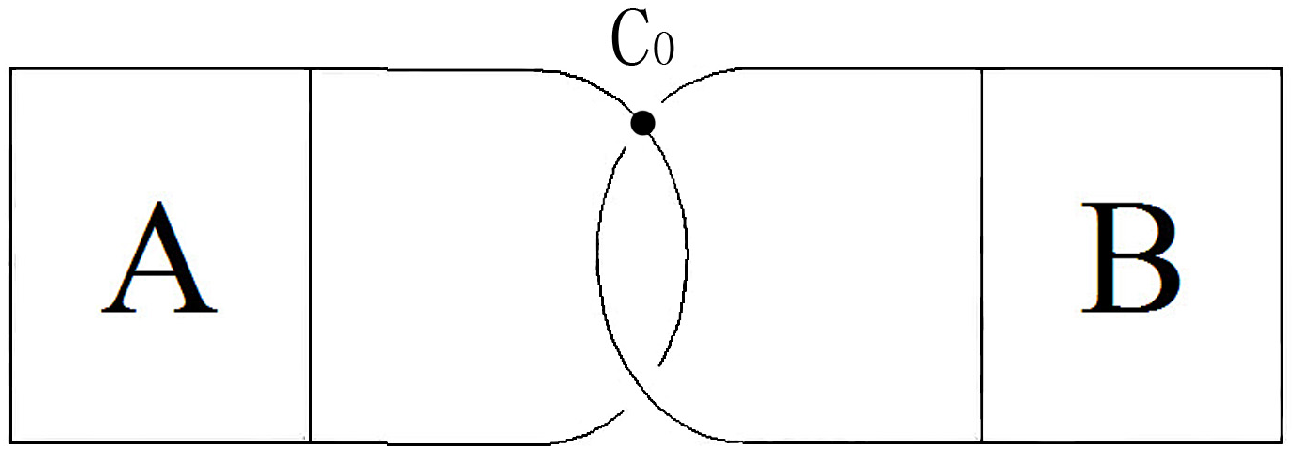}}\\
\textbf{Figure 9}
\end{center}

With a crossing change to $c_0$, $L_-$ generates into $L_+$. And by
Figure 9, it is easily to know that $c(L_+)\leq c(L_-)-2$.

\item If $n>2$, See Figure 10.

\begin{center}
\scalebox{0.4}{\includegraphics*[280,200]{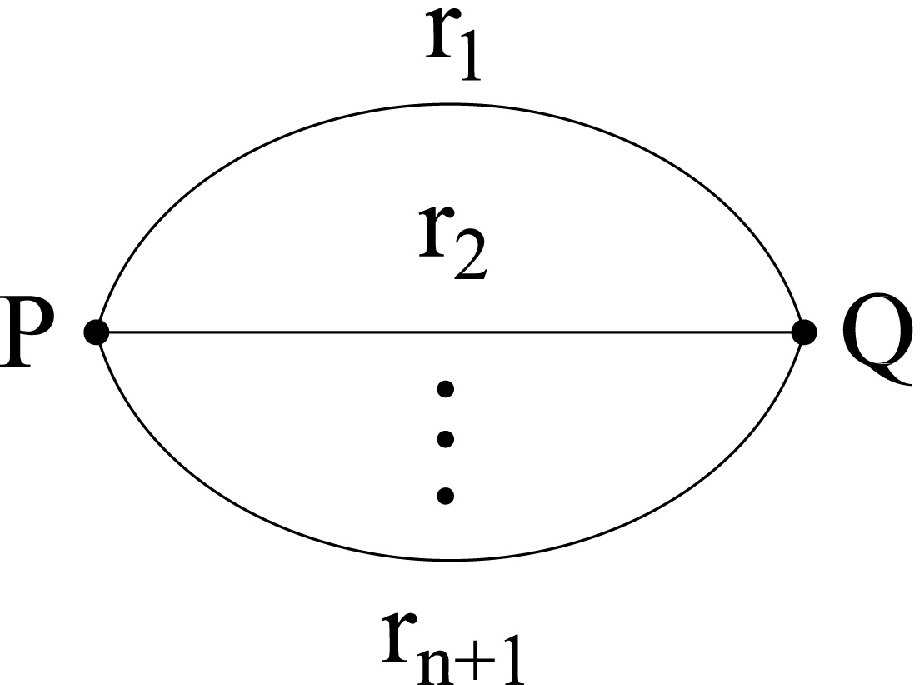}}\\
\textbf{Figure 10}
\end{center}

Then we say that $L_0$ must be reduced.

In fact, if $L_0$ is not reduced, then based on the inference of the
situation of $n=2$, we can prove that there are 2 regions, $r_i$ and
$r_{i+2}$, in Figure 10 that are connected, but it can be easily see
from Figure 10 that it is impossible
 when $n>2$~! Thus, the situation of $n>2$ could be boiled down to
case 1.

\end{enumerate}
\end{enumerate}
\end{proof}

\setcounter{section}{3}
\section{Examples and future work}
As a matter of fact, we have verified the theorem to all the
suitable knot projections with crossing number less than 11.

Next, we will demonstrate that the situation of
$c(\tilde{L})=c(L)-2$ in Theorem 1.1 exists. See
\newtheorem{example}{Example}[section]

\begin{example}
In Figure 11, we reverse which goes over and which goes under at
crossing $C$ of a reduced and alternating projection of knot 5.1.
Then, we obtain
 a projection of alternating
knot 3.1. And $c(3.1)=c(5.1)-2$.
\end{example}

\begin{picture}(320,145)
\put(65,20){\scalebox{0.3}[0.3]{\includegraphics*[800,430]{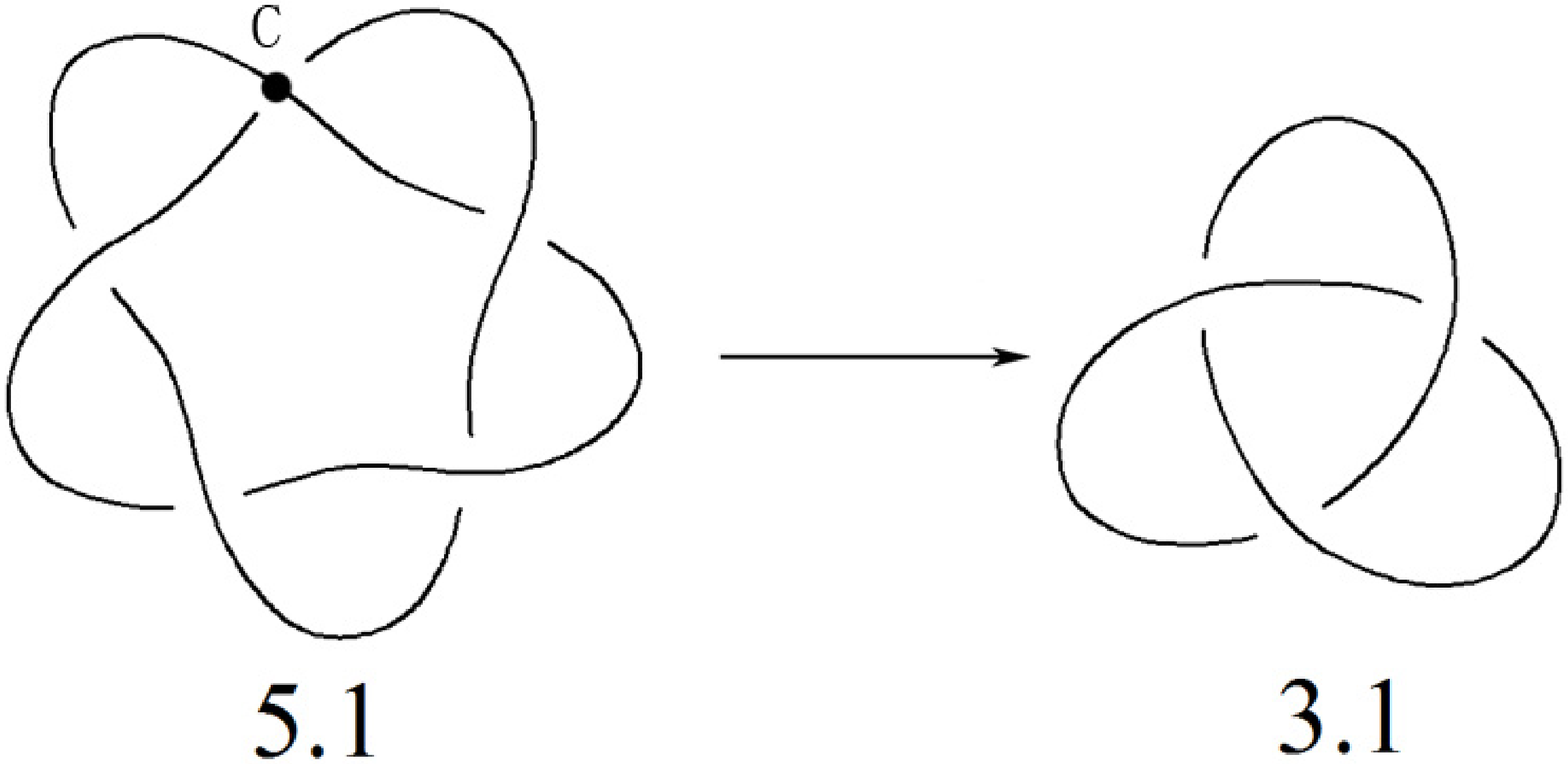}}}
\put(165,5){\textbf{Figure 11}}
\end{picture}

However, it is notable that the condition that obtaining a
projection of alternating link after changing the link in Theorem
1.1 is necessary. In fact, if we obtain a projection of
non-alternating link, the conclusion of Theorem 1.1 cannot be
established.

\begin{example}
We choose a reduced and alternating projection of knot 10.70 and
reverse which goes over and which goes under at crossing $C$, then
we get a projection of non-alternating knot 9.42. As we can see,
$c(9.42)=c(10.70)-1$.
\end{example}

\begin{picture}(320,140)
\put(70,20){\scalebox{1.28}{\includegraphics*[200,100]{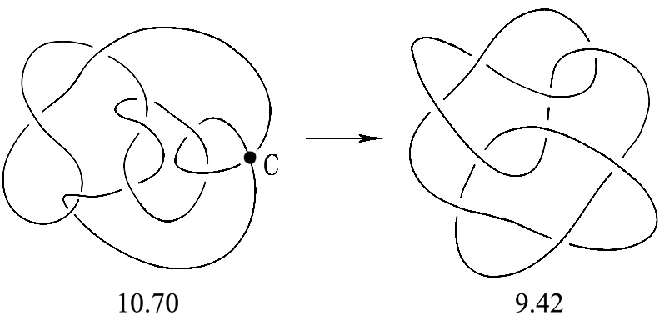}}}
\put(165,5){\textbf{Figure 12}}
\end{picture}

However, we do not know whether there exists a reduced and
alternating projection of a connected link, which can maintains its
crossing number after a crossing change. We are aiming at finding
more examples and broadening our research range to non-alternating
link in the future, so that we can explore more essential characters
of the change of crossing number under the influence of a crossing
change.


\begin{thebibliography}{11}
\bibitem[1]{sto}A. Stoimenow, Non-triviality of the Jones polynomial and the crossing numbers of amphicheiral knots, preprint, arXiv:math/0606255v2 [math.GT].
\bibitem[2]{bol}B. Bollob$\acute{a}$s, \textit{Modern graph theory}, Graduate Texts in Mathematics (Springer, New York, 1998).
\bibitem[3]{Mur}K. Murasugi, Jones polynomials and classical conjectures in knot theory, \textit{Topology} \textbf{26}(2) (1987) 187-194.
\bibitem[4]{Ka}L. H. Kauffman, State models and the Jones polynomial, \textit{Topology} \textbf{26}(3) (1987) 395-407.
\bibitem[5]{thist87} M. B. Thistlethwaite, A spanning tree expansion of the Jones polynomial, \textit{Topology} \textbf{26}(3) (1987) 297-309.
\bibitem[6]{scharlemann}M. Scharlemann, Crossing changes, \textit{Chaos, Solitons and Fractals} \textbf{9} (1998) 693-704.
\bibitem[7]{DL}O. Dasbach and X. Lin, A volume-ish theorem for the Jones polynomial of alternating knots, \textit{Pacific J. Math.} \textbf{231}(2) (2007) 279-291.
\bibitem[8]{lin}R. Kirby, Problems in low dimensional topology, in \textit{Geometric
Topol., Proc. Georgia International Topology Conf.} (University of
Georgia, Athens, Georgia, 1993), ed. W. Kazez, AMS/IP Studies in
Advanced Mathematics, Vol. 2, Part 2, pp. 35-473 (Amer. Math. Soc.,
Providence, R.I., 1997).
%ok..ok...ºÜÖØÒªÂð£¿
\bibitem[9]{jones}V. F. R. Jones, A polynomial invariant for knots via von Neumann algebras, \textit{Bull. Amer. Math.
Soc.} \textbf{12}(1) (1985) 103-111.
\bibitem[10]{thist}W. B. R. Lickorish and M. B. Thistlethwaite, Some links with non-trivial polynomials
and their crossing numbers, \textit{Commet. Math. Helv.} \textbf{63}
(1988) 527-539.
\end{thebibliography}
\end{document}